\documentclass[11pt]{amsart}
\usepackage{array, amsmath, amscd, amssymb, amsthm}
\usepackage{hyperref}
\newcommand{\inclu}{\hookrightarrow}

\newcommand{\cXn}{\cX^{[n]}}
\newcommand{\Hn}[1]{{#1}^{[n]}}
\newcommand{\Hilb}[2]{{#1}^{[#2]}}

\newcommand{\cLn}{\cL^{[n]}}

\newcommand{\Gott}{G\"{o}ttsche}

\newcommand{\fourtop}{L^2, LK, c_1(S)^2, c_2(S)}
\newcommand{\fourtopand}{$L^2$, $LK$, $c_1(S)^2$ and $c_2(S)$}

\newcommand{\ua}{\underline{\alpha}}
\newcommand{\uan}{(\alpha_1,\alpha_2,\ldots, \alpha_{l(\ua)})}
\newcommand{\dd}[2]{\frac{\partial #1}{\partial #2} }

\newcommand{\ddy}[1]{\frac{\partial #1}{\partial y} }
\newcommand{\Cxy}{\mathbb{C}\{x,y\}}
\newcommand{\codim}{\text{codim}}

\theoremstyle{plain}

\newtheorem{prop}{Proposition}[section]
\newtheorem{theo}[prop]{Theorem}
\newtheorem{lemm}[prop]{Lemma}
\newtheorem{coro}[prop]{Corollary}

\theoremstyle{definition}
\newtheorem{defn}{Definition}[section]

\newtheorem{eg}{Example}[section]

\theoremstyle{remark}
\newtheorem*{rem}{Remark}
\newtheorem*{notation}{Notation}

\newtheorem*{claim}{Claim}

\newcommand{\CC}{\mathbb{C}}

\newcommand{\NN}{\mathbb{N}}
\newcommand{\PP}{\mathbb{P}}
\newcommand{\QQ}{\mathbb{Q}}

\def\cC{{\mathcal C}}

\def\cK{{\mathcal K}}
\def\cL{{\mathcal L}}

\def\cO{{\mathcal O}}

\def\cX{{\mathcal X}}

\def\fm{\mathfrak{m}}

\let\oldTitle\title
\renewcommand{\title}[1]{\newcommand{\myTitle}{#1}\oldTitle{#1}} 
\title{Universal polynomials for singular curves on surfaces}

\input xy
\xyoption {all}
\numberwithin{equation}{section}
\begin{document}

\author{Jun Li}
\address{Department of Mathematics, Stanford University,
California, CA 94305, USA}
\email{jli@math.stanford.edu}

\author{Yu-jong Tzeng}
\address{Department of Mathematics, Harvard University,
Cambridge, MA 02138, USA} 
\email{ytzeng@math.harvard.edu}

\begin{abstract}  
Let $S$ be a complex smooth projective surface and $L$ be a line bundle on $S$. 
For any given collection of isolated topological or analytic singularity types, we show the  number of curves in the linear system $|L|$ with prescribed singularities is a universal polynomial of Chern numbers of $L$ and $S$, assuming $L$ is sufficiently ample. 
%Here the number of singular points can be arbitrary and the singularity types can be chosen to be either analytic or topological equivalent. 
Moreover, we define a generating series whose coefficients are these universal polynomials and discuss its properties.  This work is a generalization of  \Gott's conjecture to curves with higher singularities.  
  
\end{abstract}
\maketitle

\section{Introduction}
For a pair of a smooth projective surface and a line bundle $(S,L)$, it is a classical problem to find the number of $r$-nodal curves in a generic $r$-dimensional linear subsystem of $|L|$. 
\Gott\ conjectured that for any $r\geq0$, there exists a universal polynomial $T_r$ of degree $r$, such that $T_r(L^2, LK_S, c_1(S)^2, c_2(S))$ equals the number of $r$-nodal curves in a general linear subsystem, provided that $L$ is $(5r-1)$-very ample. 
Moreover, the generating series of $T_r$ has a  multiplicative structure and satisfies the \Gott-Yau-Zaslow formula (\cite{Gott}, \cite{Tz}). 
Recently, \Gott's universality conjecture was proven  by the second named author \cite{Tz} using degeneration methods, and a different proof was given by Kool-Shende-Thomas \cite{KST} using BPS calculus and computation of tautological integrals on Hilbert schemes (see also the approach of Liu \cite{Liu1} \cite{Liu2}). In this paper, we address the question whether a similar phenomenon is true for curves with higher singularities? 

The goal of this article is to generalize \Gott's universality conjecture to curves with arbitrary isolated (analytic or topological) singularities. Consider a collection  of isolated singularity type  $\ua=\uan$. We say a curve $C$ has singularity type $\ua$ if there exists $l(\ua)$ points $x_1$, $x_2$,\ldots $x_{l(\ua)}$, such that the singularity type of $C$ at $x_i$ are exactly $\alpha_i$ and $C$ has no more singular points. 
We prove the following theorem concerning curves with singularity type $\ua$: 
\begin{theo} For every collection of isolated singularity type  $\ua$, there exists a universal polynomial $T_{\ua}(x,y,z,t)$ of degree $l(\ua)$ with the following property: given a smooth projective surface $S$ and an $(N(\ua)+2)$-very ample  line bundle $L$ on $S$, a general $\codim(\ua)$-dimensional sublinear system of $|L|$ contains exactly $T_{\ua}(L^2, LK, c_1(S)^2, c_2(S))$ curves with singularity type $\ua$.  
\end{theo}
For instance, there is a universal polynomial which counts curves with one triple point, two $E_8$ singularity, and one $5$-fold point analytic equivalent to $x^5-y^5=0$. 

To characterize the conditions of given singularity type, we study the locus of zero-dimensional closed subschemes of a special shape on the surface. The shape is determined by the singularity type such that if a curve has a prescribed singularity then it must contain a zero-dimensional closed subscheme of the corresponding shape; moreover, the converse is true for generic curves.  In the case of node, the locus is the collection of subschemes isomorphic to $\Cxy/\fm^2$ because a curve $C$ is singular at $p$ if and only if $C$  contains $\cO_p/{\fm_p^2}$. 
This technique was developed by \cite{Gott} and \cite{HP} to show that the enumeration of nodal curves can be achieved by computing certain intersection number on Hilbert schemes of points. 
%However, it is extremely difficult to compute the intersection number directly. 
In this paper, we find a uniform way of defining such correspondence from  isolated singularity types to the punctual Hilbert schemes.
As a result, the number of curves with given singularity types can be expressed again as intersection numbers on Hilbert schemes of points on $S$. 
Then we apply a degeneration argument developed in \cite{Tz} to show the existence of universal polynomials. 

For example, the number of cuspidal curves can be computed in the following steps: 

\begin{enumerate}
\item If a curve $C$ has a cusp at a point $p$, then we can choose local coordinates at $p$ such that $C$ is defined by $y^2-x^3=0$. 
Therefore $C$ contains  a subscheme isomorphic to  $\Cxy/\langle y^2-x^3 \rangle$ supported at $p$. 
But since we want a zero-dimensional subscheme of finite length, we can quotient further by $\fm^4$ because $\fm^4$ will not affect the cuspidal condition. 

\item Define $S^0(\text{cusp})$ to be all closed subschemes isomorphic to \mbox{$\Cxy/\langle y^2-x^3, \fm^4 \rangle$} in $\Hilb{S}{7}$, take $S(\text{cusp})$ to be the closure of $S^0(\text{cusp})$ (with induced reduced structure). It is elementary to see that the dimensions of  $S^0(\text{cusp})$ and $S(\text{cusp})$ are both $5$. 

\item Since the locus of cuspidal curves is of codimension two in $|L|$, our goal is  computing the (finite) number of cuspidal curves in a general linear subsystem $V\cong\PP^2\subset |L|$.  Suppose that  $L$ is sufficiently ample and let $L^{[7]}$ be the tautological bundle of $L$ in $S^{[7]}$, 
the number of cuspidal curves is equal to  the number of points of the locus cut out by the sections in $H^0(L^{[7]})$ induced by $V$ in $S(\text{cusp})$; when the intersection is discrete, this number is represented by $d_{\text{cusp}} (S,L):= \int_{S(\text{cusp})} c_{7-3+1} (L^{[7]})$ (by applying the Thom-Porteous Formula to the locus where three general sections of $L^{[7]}$ are linearly dependent). 

\item The degree of the zero cycle $d_{\text{cusp}} (S,L)$ does not have contribution from nonreduced curves, curves with more than two singular points or with singularity worse than cusp. Because if a curve in $d_{\text{cusp}} (S,L)$ is nonreduced or has more than two singular points, then it  must contain  $\cO_{S,p}/\langle y^2-x^3, {\fm_{S,p}}^4 \rangle \amalg \cO_{S,q}/ \fm_{S,q}^2$ for some points $p\neq q$ in $S$. Similarly a curve in $d_{\text{cusp}} (S,L)$ with singularity worse than cusp must contain $\Cxy/\langle (y^2-x^3)\fm, \fm^4\rangle$. 
Dimension count shows that this is impossible for general $V$ and sufficiently ample $L$. 

\item In the last step, we apply the degeneration technique developed in \cite{Tz}.  We show $d_{\text{cusp}}(S,L)$ only depend on the class of $[S,L]$ in the algebraic cobordism group, which only depends on \fourtopand.  
Hence $d_{\text{cusp}}(S,L)$ is a polynomial in \fourtopand\ and  is exactly the universal polynomial we are looking for. 

\end{enumerate}
 
Our method does not provide a constructive way for  computing the  universal polynomials. However, the multiplicative property of generating series (Theorem \ref{thm:Ai}) imposes strong restrictions on universal polynomials. In the end of Section \ref{sec:3}, we discuss a few known cases and the relation with the Thom polynomials. Moreover, we will discuss the irreducibility and smoothness of the locus of curves with fixed singularity type $\ua$ in  Section 4.  

Recently we realized that our method can be generalized to count singular hypersurfaces and more objects in higher dimensional smooth varities.  The result will appear in a subsequent paper. 

\subsection*{Acknowledgments} The first named author is partially supported by an NSF grant NSF0601002 and the second named author is supported by the Simons Postdoctoral Fellowship. The second named author would like to thank Professor Harris and Professor Kleiman for useful conversations and Professor S.T. Yau for his support and encouragements. We thank Vivek Shende for helpful discussion, especially the argument about the irreducibility of Severi strata in Section \ref{sec:irr}.

\section{The construction of $d_{\ua}(S,L)$}
In this section we  construct a zero cycle $d_{\ua}(S,L)$ on Hilbert schemes of points on $S$, and prove that this cycle can be used to count the number of curves with prescribed singularity $\ua$ in $L$. 
In order to  construct the cycle,  let us first recall some results in singularity theory. 

Let $\CC\{x,y\}$ be the local ring of the origin on  $\CC^2$,   $f$ be a germ in $\CC\{x,y\}$, and the Jacobian ideal $\langle  \dd{f}{x}, \ddy{f} \rangle$ be $J(f)$, then the  Milnor number $\mu(f)$ and Tjurina number  $\tau(f)$  are defined as:
$$ \mu(f)=\text{dim}_{\CC} \CC\{x,y\}/  J(f);\ \    \tau(f)=\text{dim}_{\CC} \CC\{x,y\}/\langle f+ J(f)\rangle.$$

Two planar curves $C_1$ and $C_2$ have \textit{analytic equivalent} singularities at the origin if their defining germs $f_1$ and $f_2$ in $\Cxy$ are contact equivalent; i.e. there exists  an automorphism $\phi$ of $\Cxy$ and a unit $u \in \Cxy$ such that 
$f = u\cdot \phi(g)$. 
We say they have \textit{topological equivalent} (or \textit{equisingular}) singularities  if the following equivalent conditions are satisfied (\cite{Gr}): 

\begin{enumerate}
\item there exists balls $B_1$ and $B_2$ with center $0$ such that 
$(B_1,B_1\cap C_1,0) $ is homeomorphic to  $(B_2,B_2\cap C_2,0) $,
\item $C_1$ and $C_2$ have the same number of branches, Puiseux pairs of their branches  $C_{1i}$ and $C_{2i}$ coincide, and intersection multiplicities $i(C_{1i},C_{1j})=i(C_{2i},C_{2j})$ for any $i$, $j$. 
\item The systems of multiplicity sequences of an embedded resolution coincide.
\end{enumerate}

It is easy to see that analytic equivalence implies topological equivalence. However,  the converse is true only for ADE singularities. The Milnor number is both analytic and topological invariant (by a result of Milnor), while the Tjurina number is only an analytic invariant. Therefore $\tau(\alpha)$ and  $\mu(\alpha)$ are well-defined when they are invariants of the singularity $\alpha$.

Recall that the tangent space of the miniversal deformation space $\mathsf{Def}$ of the singular curve $C:=\{f=0\}$ at the origin can be naturally identified with $ \CC\{x,y\}/\langle f+ J(f)\rangle $, and its dimension is the Tjurina number $\tau(f)$.   
The dimension of the miniversal deformation space also has an important geometric meaning: it is the expected codimension of the locus of curves in a linear system with same analytic singularity $\alpha$ of $C$ at the origin, and is denoted by $\text{codim}(\alpha)$. Roughly speaking, $\text{codim}(\alpha)$ is the number of conditions needed to have the singularity $\alpha$.
If $\alpha$ is a topological singularity type, then inside $\mathsf{Def}$ there is the \text{equisingular locus} $ES$,  which parametrizes equisingular or topologically trivial deformations. In this case, $\text{codim}(\alpha)$ is $\text{dim}_{\CC} \CC\{x,y\}/\langle f+ J(f)\rangle)-\text{dim}_{\CC} ES$ and  is also the expected codimension of curvesㄙ with singularity $\alpha$ in a linear system. 
%In both cases,   $\text{codim}(\alpha)$ equals the codimension of the locally closed locus in $|L|$ which parametrizes curves with one singular point of type $\alpha$, assuming $L$ is a sufficiently ample line bundle.

A natural question is,  in order to determine the singularity type of  $f=0$ at the origin, is it sufficient to look at some lower degree terms of $f$? For example, does the curve $y^2=x^3+xy^{99}$ have  a cusp because the lower degree terms of $f$ is  $y^2=x^3$? The answer is yes, provided if the terms being ignored are of sufficiently high degrees, according to the finite-determinacy theorem:

\begin{defn} We call a germ  $f$ to be (analytically) $k$-determined if 
$f \equiv g\ \ (\text{mod } \fm^{k+1}) $ implies the curves $f=0$ and $g=0$ have analytic equivalent singularities at the origin. 

\end{defn} 

%Here is a simple criterion for finite determinacy:
\begin{theo}[\cite{GLS}, theorem 2.23]\label{thm:GLS}
 If $f$ is (analytically) $k$-determined, then $$ \fm^{k+1} \subseteq \fm J(f) + \langle f \rangle.$$ 
 Conversely, if $$\fm^{k} \subseteq \fm J(f) + \langle f \rangle,$$ then $f$ is $k$-determined. 
\end{theo}
\begin{coro}[\cite{GLS}, corollary 2.24]
For any germ $f\in \fm \subset \Cxy$, $f$ is $\tau(f)$-determined.\end{coro}

Since all representatives of an analytical (resp. topological) singularity $\alpha$ are $\tau(\alpha)$-determined (resp. $\mu(\alpha)$-determined), we can define $k(\alpha)$ to be the smallest $k$ such that all representatives of $\alpha$ are $k$-determined.

%If $\alpha$ is an equisingular type, then we can still choose a germ $f_{\alpha} \in \CC\{x,y\}$ to represent $\alpha$. Then $\tau(\alpha)$ is still defined in the same way but $\text{codim}(\alpha)$ is defined to be $\tau(\alpha)-\text{dim} \mathbf{ES}$, where $\mathbf{ES}$ is the equisingular locus in the versal deformation space. Since the equisingular locus is cut out by an equisingular ideal $I$ which contains $f$ and $Jf$, $\text{codim}(\alpha)=\text{dim} \CC{x,y}/I$.  Similarly, $\text{codim}(\alpha)$ is the codimension of the locally closed subscheme in $|L|$ which parametrizes curves with one singular point of equisingular type $\alpha$,  if $L$ is sufficiently ample. 

%Topological singularity types are still finitely determined (\cite{Wa}, theorem 11.5.6). 
%For a topological singularity type $\alpha$,  the smallest $k$ such that \textit{every} representative of is $k$-determined is denoted by $k(\alpha)$.  

\begin{defn} For any isolated planar (analytic or topological) singularity $\alpha$, pick  a representative $f_{\alpha}\in \Cxy$ and let $N(\alpha)$ be the length of the zero-dimensional closed subscheme $\xi_{\alpha}=\CC\{x,y\}/\langle f_{\alpha}, \fm^{k(\alpha)+1}\rangle $.    
\end{defn}
It is easy to check that the number $N(\alpha)$ does not depend on the choice of $f_{\alpha}$.

\begin{eg} If $\alpha$ is a simple node,  we can choose $f_{\alpha}=xy$. Then $\tau(\alpha)=\text{codim}(\alpha)=1$, $k(\alpha)=2$, $\xi_{\alpha}=\CC\{x,y\}/\langle xy,m^3\rangle $ and  $N(\alpha)=5$.
\end{eg}

\begin{eg} If $\alpha$ is an ordinary cusp, we can choose $f_{\alpha}=y^2-x^3$. Then $\tau(\alpha)=\text{codim}(\alpha)=2$, $k(\alpha)=3$, $\xi_{\alpha}=\CC\{x,y\}/\langle y^2-x^3,m^4\rangle $ and  $N(\alpha)=7$.
\end{eg}

\begin{eg}\label{eg:ana-n-fold} If $\alpha$ is an analytical $n$-fold point which is defined by  $f_{\alpha}=x^n-y^n$, then $\tau(\alpha)=\text{codim}(\alpha)=(n-1)^2=\text{dim}_{\CC} \CC[x^iy^j], 0\leq i,j \leq n-2$, $k(\alpha)=n$, $\xi_{\alpha}=\CC\{x,y\}/\langle x^n-y^n, \fm^{n+1}\rangle $ and  $N(\alpha)=\frac{(n+1)(n+2)}{2}-1=\frac{n(n+1)}{2}$.
\end{eg}

\begin{eg}\label{eg:top-n-fold} If $\alpha$ is a topological $n$-fold point ($n\geq 3$), we can  choose $f_{\alpha}=x^n-y^n$ again. 
Then $\tau(f_{\alpha})=(n-1)^2$, $\text{codim}(\alpha)=(n-1)^2-(n-3)$ (because an analytic isomorphism can only send three branches to $x=0$, $y=0$ and $x=y$, so $\text{dim}_{\CC} ES=n-3$, corresponding to the slope of the rest $n-3$ branches), $k(\alpha)=n$, $\xi_{\alpha}=\CC\{x,y\}/\langle x^n-y^n,\fm^{n+1}\rangle $ and $N(\alpha)=\frac{(n+1)(n+2)}{2}-1=\frac{n(n+1)}{2}$.
\end{eg}

Recall that if  $\ua=(\alpha_1, \alpha_2,\ldots, \alpha_{l(\ua)})$ is  a collection of isolated planar (analytic or topological) singularity types.  A curve $C$ has singularity type $\ua$ if $C$ is singular at exactly $l(\ua)$ distinct points $\{x_1,x_2,\ldots, x_{l(\ua)}\}$ and the  singularity type at $x_i$ is $\alpha_i$. Note we consider the topological and analytic singularities defined by the same germ as different singularities. 

Define 
\begin{align*}
 N(\ua)=\sum_{i=1}^{l(\ua)} N(\alpha_i), \,\,\text{ and} \,\,
\text{codim}(\ua)=\sum_{i=1}^{l(\ua)} \text{codim}(\alpha_i).\ \ 
\end{align*}
We comment that $\text{codim}(\ua)$ is the expected codimension of the locus in $|L|$ which parametrizes curves with singularity types $\ua$. In Proposition \ref{prop:daL}, we will prove that this locus is nonempty and of expected codimension if $L$ is $(N(\ua)+2)$-very ample. 
 
For a smooth surface $S$, let $S^{[N(\ua)]}$ be the Hilbert scheme of $N(\ua)$ points on $S$. Define $S^0(\ua)\subset S^{[N(\ua)]}$ to be the set of points $  \coprod_{i=1}^{l(\ua)} \eta_i$ satisfying the following conditions: 
\begin{enumerate}
\item $\eta_i$'s are supported on distinct points of $S$;

%\item if $\alpha_i$ is an analytic singularity type, then $\eta_i$ is  analytically isomorphic to  $\xi_{\alpha_i}=\CC\{x,y\}/\langle f_{\alpha_i} +\fm^{k_{\alpha_i}+1} \rangle$. %,

\item  $\eta_i$ is isomorphic  to $\CC\{x,y\}/\langle g_i, \fm^{k(\alpha_i)+1} \rangle$, for a germ $g_i$ such that $g_i=0$ has singularity type $\alpha_i$ at the origin. 
\end{enumerate}

Consider the closure $S(\ua)=\overline{S^0(\ua)}$ as a closed subscheme in $S^{[N(\ua)]}$. For every $n\in \NN$, let 
$Z_n\subset S\times S^{[n]}$ be the universal closed subscheme with projections 
$p_n : Z_n \to S$, $q_n : Z_n \to S^{[n]}$. 
If $L$ is a line bundle on $S$, define $L^{[n]} = {(q_n)}_* (p_n)^*L$. Because $q_n$ is finite and flat, $L^{[n]}$ is a vector bundle of rank $n$ on $S^{[n]}$.
To count curves with singularity type $\ua$, we use the cycle 
$$d_{\ua}(S, L)=\int_{S(\ua)}  c_{N(\ua)-\text{codim}(\ua)}(L^{[N(\ua)]}).$$

\begin{lemm} The term $d_{\ua}(S, L)$ is a zero cycle;  i.e. $\text{dim }S(\ua)=N(\ua)-\text{codim}(\ua)$. 
\end{lemm}
\begin{proof}
It suffices to prove that for every isolated singularity $\alpha$, the dimension of $S^0(\alpha)$ is equal to $N(\alpha)-\text{codim}(\alpha)$. 

Suppose that $\alpha$ is an analytic singularity, by definition every closed subscheme in $S^0(\alpha)$ is only supported at one point on $S$. Define the projection $p: S^0(\alpha)\to S$ to be the map sending a closed subscheme to its support, then the fiber over a point  is the collection of closed subschemes $\CC\{x,y\}/\langle g, \fm^{k(\alpha)+1} \rangle$ supported at that point such that  $g=0$ has singularity $\alpha$. If we pick a representative $f$, all such $g$ is in the orbit of $f$ under the group action by 
$\cK={\Cxy}^* \ltimes  Aut (\Cxy)$.  
Let $orb(f)$ be the orbit of $f$ in $\Cxy/\fm^{k(\alpha)+1}$ under the action of the restriction of $\cK$ on the $\Cxy/\fm^{k(\alpha)+1}$. 
According to \cite{Gr}, $orb(f)$ is smooth and its tangent space at $f$ is 􏰂$$(\fm \cdot J(f) + \langle f \rangle + \fm^{k(\alpha)+1}􏰃) /\fm^{k(\alpha)+1}􏰃.$$ 
But since $\CC\{x,y\}/\langle f, \fm^{k(\alpha)+1} \rangle$ and $\CC\{x,y\}/\langle u\cdot f, \fm^{k(\alpha)+1} \rangle$ described the same closed subscheme if $u$ is a unit in $\Cxy/\fm^{k(\alpha)+1-m(f)}$ (where $m(f)$ is the multiplicity of $f$ at the origin), $orb(f)/(\Cxy/\fm^{k(\alpha)+1-m(f)})^*$ is isomorphic to the fiber of $p$ over every point on $S$. 

The discussion above and Theorem \ref{thm:GLS} imply
\begin{align*}
&\text{dim}_{\CC} S^0(\alpha)\\
=& 
2+ \text{dim}_{\CC} (\fm \cdot J(f) + \langle f \rangle + \fm^{k(\alpha)+1}) /\fm^{k(\alpha)+1}-\text{dim}_{\CC} \Cxy/\fm^{k(\alpha)+1-m(f)}\\
=& \text{dim}_{\CC} (J(f) + \langle f \rangle ) /\fm^{k(\alpha)+1}􏰃-\text{dim}_{\CC} \Cxy/\fm^{k(\alpha)+1-m(f)}\\
=&\text{dim}_{\CC}  \Cxy/\fm^{k(\alpha)+1}􏰃
-\text{dim}_{\CC}  \Cxy/( J(f) + \langle f \rangle 􏰃)-
\text{dim}_{\CC} \Cxy/\fm^{k(\alpha)+1-m(f)}\\
=&\left(\text{dim}_{\CC}  \Cxy/\langle f, \fm^{k(\alpha)+1}􏰃 \rangle
+ \text{dim}_{\CC} \langle f, \fm^{k(\alpha)+1}􏰃\rangle/\fm^{k(\alpha)+1}􏰃\right)\\
&-\text{codim}(\alpha)-\text{dim}_{\CC} \Cxy/\fm^{k(\alpha)+1-m(f)}\\
=&N(\alpha)+\text{dim}_{\CC} \Cxy/\fm^{k(\alpha)+1-m(f)}-\text{codim}(\alpha)-\text{dim}_{\CC} \Cxy/\fm^{k(\alpha)+1-m(f)}\\
=&N(\alpha)-\text{codim}(\alpha).
\end{align*}

If  $f=0$ defines an analytic singularity $\alpha$ and a topological singularity $\beta$ at the origin, it follows from definition that 
$\text{dim}_{\CC} S^0(\beta)=\text{dim}_{\CC} S^0(\beta)+\text{dim}_{\CC} ES$,
$N(\alpha)=N(\beta)$, and 
$\text{codim}(\beta)=\text{codim}(\alpha)-\text{dim}_{\CC} ES$. 
Therefore the desired equality is established for topological singularity types. 
\end{proof}

\begin{prop}\label{prop:daL} Assume L is $(N(\ua)+2)$-very ample, then a general linear subsystem $V\subset |L|$ of dimension $\text{codim}(\ua)$ contains precisely $d_{\ua}(S, L)$ curves  whose singularity types are $\ua$. 
\end{prop}
\begin{proof}
The structure of the proof is essentially the same as (\cite{Gott}, proposition 5.2). However, we have to make necessary generalization to deal with all singularity types. 

Because $L$ is $N(\ua)$-very ample, $H^0(L)\to L|_{\xi}$ is surjective for every $\xi\in \Hilb{S}{N(\ua)}$. If $\{s_i\}$ is a basis of $V$, then 
$\{(q_n)_*(p_n)^* s_i\}$ are global sections in $H^0(L^{[N(\ua)]})$. 
For general $V$, the cycle of the locus $W$ where $\{(q_n)_*(p_n)^* s_i\}$ are linearly dependent on $S(\ua)$ can be expressed as $[W]=d_{\ua}(S,L)=\int_{S(\ua) }c_{N(\ua)-\text{codim}(\ua)}(\Hilb{L}{N(\ua)})$ by the Thom-Porteous formula and (\cite{Fu}, example 14.4.2). 
Applying the Thom-Porteous formula again to $S(\ua)\backslash S^0(\ua)$,  we conclude that $W$  only supports on $S^0(\ua)$. By an argument similar to (\cite{Gott}, proposition 5.2),  $W$ is smooth for general $V$. Therefore $d_{\ua}(S,L)$ is the number of curves in a general   $V$ which contains a point in $S^0(\ua)$.

Next, we show curves in $d_{\ua}(S,L)$ can not have more than $l(\ua)+1$ singular points. Assume $L$ is $(N(\ua)+2)$-very ample and suppose that there is a curve $C$ in $d_{\ua}(S,L)$ with $l(\ua)+1$ singular points.  
Then $C$ must contain a point in 
$$S'(\ua)=\left\{\displaystyle \coprod_{i=1}^{l(\ua)+1} \eta_i  \,\left|\,\,
 \coprod_{i=1}^{l(\ua)} \eta_i \in S^0(\ua) \text{ and } \eta_{l(\ua)+1}\cong
\Cxy/\fm^2 \right. \right\}\subset \Hilb{S}{N(\ua)+3}.$$
Apply the Thom-Porteous formula to the closure of $S'(\ua)$ and a dimension count proves that there is no such curve in a general $V\subset |L|$ of dimension $\codim(\ua)$. It also follows that $C$ must be reduced. 

Finally, we prove that curves in $d_{\ua}(S,L)$ must have singularities  type precisely $\ua$. A curve $C$ in $d_{\ua}(S,L)$ must contain a point 
$\coprod_{i=1}^{l(\ua)} \eta_i $ in $S^0(\ua)$ and hence has singularity 	``at least'' $\alpha_i$ at $x_i$, where  $x_i$ is the support of $\eta_i$. 
If the singularity type of $C$ is not $\ua$, without loss of generality we can assume the singularity type at $x_1$ is not $\alpha_1$. 
Let $\eta_1=\Cxy / \langle g_1, \fm^{k(\alpha_1)+1}\rangle$ and the germ of $C$ at  $x_1$ be $f$. 
Since $C$ contains $\eta_1$, $f\in \langle g_1, \fm^{k(\alpha_1)+1}\rangle$ and thus $f\equiv u\cdot g_1$ (mod $\fm^{k(\alpha_1)+1}$). 
If $u$ is a unit, then $uf$ also defines $C$ and the finite determinacy theorem implies $C$ must have singularity precisely $\alpha_1$ and this is a contradiction. 
Otherwise, $u$ is not a unit and $f$ is in the ideal  $\langle g_1\fm, \fm^{k(\alpha_1)+1}\rangle$.  Let $\tilde{S}^0(\ua)$ be the set of $ \Cxy/\langle g_1\fm, \fm^{k(\alpha_1)+1}\rangle
\cup  \left(\coprod_{i=2}^{l(\ua)}\eta_i \right)$ such that  there exists $\eta_1=\Cxy/\langle g_1, \fm^{k(\alpha_1)+1}\rangle$ and 
$\coprod_{i=1}^{l(\ua)}\eta_i \in S^0(\ua)$. 
The natural map from $S^0(\ua)$ to $\tilde{S}^0(\ua)$, which sends $\coprod_{i=1}^{l(\ua)}\eta_i$ to  $\Cxy/\langle g_1\fm, \fm^{k(\alpha_1)+1}\rangle
\cup  \left(\coprod_{i=2}^{l(\ua)}\eta_i \right)$, is surjective and therefore $\text{dim } S^0(\ua)\geq \text{dim } \tilde{S}^0(\ua)$. 
Let $\tilde{S}(\ua)$ be the closure of $\tilde{S}^0(\ua)$ in $\Hilb{S}{N(\ua)+1}$ and apply the Thom-Porteous formula to $\tilde{S}(\ua)$, we see $C$ contributes a positive number in the counting $\int_{\tilde{S} (\ua)} c_{N(\ua)+1-\codim(\ua)}(\Hilb{L}{N(\ua)+1})$\footnote{It can be checked that $g_1$ is never contained in $\fm^{k(\ua)+1}$. 
So the length of  $\Cxy/\langle g_1\fm, \fm^{k(\alpha_1)+1}\rangle$ is the length of $\Cxy/\langle g_1, \fm^{k(\alpha_1)+1}\rangle$ $-1$.}.  On the other hand, this count is zero by dimension reason. 
It leads to a contradiction and this proves that $C$ must have singularity type precisely $\ua$.  
\end{proof}

\begin{rem} Kleiman pointed out to us that  one can associate every topological singularity $\alpha$ to a complete ideal $I_{\alpha}$, such that a general element in $I_{\alpha}$ defines a curve with singularity $\alpha$. 
Therefore we can also associate $\alpha$ to  $\Cxy/I_{\alpha}$ in our approach. This may weaken the ampleness condition needed in the Proposition \ref{prop:daL} and in Theorem \ref{thm:univ}. \end{rem}

\section{Main results}\label{sec:3}
In this section, we prove the existence of universal polynomial $T_{\ua}$ that  counts the number of curves of singularity type $\ua$ in a sufficiently ample linear system. Moreover, we construct the generating series containing all $T_{\ua}$ and show this series has a very special form and is multiplicative in \fourtopand. 
%This generating series has a closed relationship with Thom polynomials. 

We assign a formal variable $x_{\alpha}$ to every isolated singularity type $\alpha$, and   define $x_{\ua}=\prod_{i=1}^{l(\ua)} x_{\alpha_i}$ if $\ua=(\alpha_1, \alpha_2,\ldots, \alpha_{l(\ua)})$. The multiplication  $x_{\ua'}\cdot x_{\ua''}$ is equal to $x_{\ua}$ if and only if $\ua$ is the union of $\ua'$ and $\ua''$. The multiplication is commutative because permutations of $\alpha_i$ does not change $\ua$.  

\begin{defn}
For a line bundle $L$ on $S$ ($L$ does not need to be ample), we put $d_{\ua}(S,L)=1$ if $\ua$ is the empty set (it corresponds to the number of smooth curves in $|L|$).
Define the generating series

$$T(S,L)=\sum_{\ua} d_{\ua} (S, L) x_{\ua} .$$ 
\end{defn}
%\footnote{Recall nodes are $A_1$-singularity and cusps are $A_2$ singularity, 
%The series 
%$$T(S,L)=1+d_{node} (S,L)x_{node}+d_{2 nodes}(S,L) x_{node}^2+d_{cusp}(S,L)x_{cusp}$$}

\begin{theo}\label{thm:Ai}There exist universal power series $A_1$, $A_2$, $A_3$, $A_4$  in $\QQ[[x_\alpha]]$ such that the generating function $T(S,L)$ has the form
\begin{align}\label{eq:Ai}T(S,L)=A_1^{L^2}A_2^{LK_S}A_3^{c_1(S)^2}A_4^{c_2(S)}.\end{align}\end{theo}
\begin{proof}
The ingredients of the proof are  algebraic cobordism of pairs of smooth schemes and vector bundles, and degeneration of Quot schemes, see  \cite{LeeP}, \cite{LP} and \cite{LW}. We only need to use the  special case for pairs of  surfaces and line bundles and degeneration of Hilbert schemes of points. The theories in this  special case are summarized in  Sections 2 and 3 of \cite{Tz}. 

Let $[X_i,L_i]$ be pairs of surfaces and line bundles. Suppose 
$$[X_0, L_1]=[X_1, L_1]+[X_2, L_2]-[X_3, L_3]$$ is the double point relation obtained from a flat morphism $\pi: \cX\to \PP^1$ and a line bundle $\cL$ on $\cX$. That means $\cX$ is a smooth $3$-fold,  $X_0$ is the smooth fiber over $0$ and the fiber over $\infty$ is $X_1\cup X_2$, intersecting transversally along a smooth divisor $D$. 
Moreover,  $X_3= \PP(\cO_D\oplus N_{X_1/D})\cong \PP( N_{X_2/D}\oplus \cO_D)$ is a $\PP^1$ bundle over $D$, $L_i=\cL|_{X_i}$ for $i=0,1,2$ and $L_3$ is the pullback of $\cL|_D$ to $X_3$. 
The algebraic cobordism group of pairs of surfaces and line bundles $\omega_{2,1}$ is defined to be the formal sum of all pairs modulo double point relations.  The class of $[S,L]$ in $\omega_{2,1}$ is uniquely determined by all Chern numbers of $L$ and $S$; i.e. \fourtopand\ (\cite{LeeP}, \cite{Tz}). 

For every double point relation induced by $\pi: \cX\to \PP^1$, Li and Wu \cite{LW} proves that there exists a family of Hilbert schemes $\cXn$ over an open subset\footnote{We choose $0$, $\infty\in U \subseteq \PP^1$ to avoid possible singular fibers of $\cX\to \PP^1$ other than $\infty$.} $ U$ of  $\PP^1$ such that the fiber of $0\in U$ is $\Hn{X_0}$, and the fiber of $\infty$ is the union of products of relative Hilbert schemes
$$\bigcup_{k=0}^n \Hilb{(X_1/D)}{k}\times \Hilb{(X_2/D)}{n-k}.$$ 

The line bundle $\cL$ on $\cX$ induces a vector bundle $\Hilb{\cL}{n}$ on $\cXn$ by pulling $\cL$ back to the universal closed subscheme, and then pushing it forward to $\cXn$. The restriction of $\cLn$ on the fiber $\Hilb{X_0}{n}$ is $\Hilb{L_0}{n}$, and on $\Hilb{(X_1/D)}{k}\times \Hilb{(X_2/D)}{n-k}$ is $\Hilb{L_1}{k}\oplus\Hilb{L_2}{n-k}$. 

\begin{notation}$X_i(\ua)$, $(X_i/D)(\ua)$, $d_{\ua}(X_i, L_i)$ and  $d_{\ua}(X_i/D, L_i)$ are defined in a similar way to $S(\ua)$ and $d_{\ua}(S, L)$ on $(X_i, L_i)$ and $(X_i/D, L_i)$ . \end{notation}

\begin{claim}There is a family of closed subschemes $\cX(\ua) \subset \cX^{[N(\ua)]}$ such that 
\begin{align*} \cX(\ua) \cap X_0^{[N(\ua)]} &=X_0(\ua), \\
\cX(\ua)\cap \left({(X_1/D)^{[m]}\times (X_2/D)^{[N(\ua)-m]}}\right)
&= \bigcup (X_1/D)(\ua_1)\times (X_2/D)(\ua_2)  \end{align*}
where the sum is over all $\ua_1$ and $\ua_2$ satisfying $\ua=\ua_1\cup\ua_2$, 
$N(\ua_1)=m$, $N(\ua_2)=N(\ua)-m$.  
Furthermore, $\cX(\ua)$ is flat over $U$ via the composition  $\cX(\ua) \inclu \cX^{[N(\ua)]}\to U$.\end{claim}
\begin{proof}[proof of claim]
Let $\cX^0(\ua)$ be the union of all $\cX_t(\ua)$, for all smooth fibers $\cX_t$ over $t\in U$. Define $\cX(\ua)$ to be the closure of $\cX^0(\ua)$ in $\Hilb{\cX}{N(\ua)}$.  The properties of  $\cX(\ua)$ can be proved by using a similar argument as in (\cite{Tz}, lemma 3.8). 
\end{proof}

The restrictions of the flat $1$-cycle 
$\int_{\cX(\ua)} c_{N(\ua)-\text{codim}(\ua)}(\Hilb{L}{N(\ua)})$  on fibers over $0$ and $\infty$ are $0$-cycles of the same degree, which implies 
$$d_{\ua}(X_0, L_0)=\sum_{\ua=\ua_1\cup \ua_2} d_{\ua_1}(X_1/D, L_1)
d_{\ua_2}(X_2/D, L_2)\text{ and }$$
\begin{align}\label{eq:daL_rel}T(X_0, L_0)=T(X_1/D, L_1)T(X_2/D, L_2).
\end{align}
To derive a relation of generating series without relative series, we  apply \eqref{eq:daL_rel} to four families: $\cX$, 
the blowup of  $X_1\times \PP^1$ along $D\times\{\infty\}$, 
the blowup of  $X_2\times \PP^1$ along $D\times\{\infty\}$,  
the blowup of  $X_3\times \PP^1$ along $D\times\{\infty\}$, and multiplicity all equalities. Thus we can conclude  a double point relation 
$$[X_0, L_1]=[X_1, L_1]+[X_2, L_2]-[X_3, L_3]  \text { implies}$$
\begin{align}T(X_0, L_0)=\frac{T(X_1, L_1)T(X_2, L_2)}{T(X_3,L_3)},\end{align}
and thus $T$ induces a homomorphism from the algebraic cobordism group $\omega_{2,1}$ to $\CC[[x_\alpha]]$. Then \eqref{eq:Ai} is proved by using the fact that the  algebraic cobordism group $\omega_{2,1}$ is isomorphic to $\QQ^4$ by the morphism $[S,L]$ to $(\fourtop)$ (\cite{LeeP}, \cite{Tz}). 
\end{proof}

\begin{theo}\label{thm:univ} For every collection of isolated  singularity type  $\ua$, there exists a universal polynomial $T_{\ua}(x,y,z,t)$ of degree $l(\ua)$ with the following property: given a smooth projective surface $S$ and an $(N(\ua)+2)$-very ample  line bundle $L$ on $S$, a general $\codim(\ua)$-dimensional sublinear system of $|L|$ contains exactly $T_{\ua}(L^2, LK, c_1(S)^2, c_2(S))$ curves with singularity type precisely $\ua$.  
\end{theo}
\begin{proof} We compare the coefficient of $x_{\ua}$ in \eqref{eq:Ai}. 
The coefficient of $x_{\ua}$ in $T(S,L)$ is $d_{\ua}(S, L)$;  on the right hand the coefficient of $x_{\ua}$ can be computed by binomial expansion (note $L^2$, $LK$, $c_1(S)^2$ and $c_2(S)$ are integers) and it is a polynomial of degree $l(\ua)$.  Therefore $d_{\ua}(S,L)$ is always a universal polynomial of \fourtopand. Moreover, Proposition \ref{prop:daL} implies that $d_{\ua}(S,L)$ is the universal polynomial $T_{\ua}(\fourtop)$ counting the number of curves with singularity type $\ua$. 
\end{proof}

Although we can prove the existence of universal polynomial $T_{\ua}$ by identifying it with $d_{\ua}(S,L)$, it is very difficult to compute the degree of $d_{\ua}(S,L)$ directly. Here are some cases where the universal polynomials are known:

\begin{itemize}

\item If  $\ua$ is $r$-nodes for any $r\in \NN$,  all coefficients of the universal polynomial can be  determined by combining any two of  \cite{BL}, \cite{CH} and \cite{Vakil}, or by \cite{KST}. The explicit formula for $r\leq 8$ was proved and listed in  \cite{KP} and \cite{Va}. 

\item Kleiman and Piene (\cite{KP})  also computed $T_{\ua}$ in many cases when the codimension is low, such as 
$\ua=(D_4, A_1)$, $(D_4, A_1, A_1)$, $(D_4, A_1, A_1)$, $(D_4, A_1, A_1, A_1)$, $(D_6)$, $(D_6, A_1)$ and  $(E_7)$. 

\item When  there is only one singular point, Kerner \cite{Ke} found an algorithm to enumerate the number of plane curves with one fixed topological type singularity, provided that the normal form is known. 

\item In general, let $N$ be a general $\PP^{\codim(\ua)}$ in  $|L|$ and $M$ be the singular locus of fibers of the universal curve over $N$. We believe the universal polynomial of $\ua$ should be equal to the Thom polynomial of the multisingularity $\ua$ applied to $M\to N$ (\cite{Kaz}, section 10). However, we do not know how to establish this fact because of the lack of proof of the Thom polynomials by algebro-geometric methods. 
\end{itemize} 

Another possible way to compute universal polynomials and the generating series is to write 
$$T(S,L)=\sum_{\ua} \frac{a_{\ua}(\fourtop)x_{\ua}}{\#Aut(\ua)}.$$
By Theorem \ref{thm:Ai}, it is easy to see that $a_{\ua}$'s are linear polynomials in \fourtopand\ for all $\ua$. 
These $a_{\ua}$  are computed in \cite{Kaz} and \cite{KP} for low codimensional $\ua$. 
Since every $a_{\ua}$ has only four terms and they determine $T_{\ua}$, it might be easier to compute $a_{\ua}$'s directly.  
When $\ua$ is $r$ nodes and $r\leq8$,  $a_{\ua}$ was realized as algebraic cycles in \cite{Qv}.

\section{Irreducibility of Severi strata of singular curves}\label{sec:irr}

The locus of reduced and irreducible degree $d$ plane curves with $r$ nodes is a locally closed subset in $|\cO(d)|$ on $\PP^2$. 
Its closure is called the Severi varieties and has been studies extensively. 
Especially, it is well known that the Severi varieties are irreducible for every $d$ and $r$ \cite{HaSeveri}.
In this section we will generalize the irreducibility theorem to   curves with  analytic singularities in the linear system of a sufficiently ample line bundle. 

\begin{theo} Let $\ua=\uan$ be a collection of analytic singularity types,  $S$ be a complex smooth projective surface and $L$ be an $(N(\ua)+2)$-very ample line bundle on $S$. Define $V_0^{d,\ua}(S,L)$ to be the locally closed subset of $|L|$  parametrizing curves with singularity type exactly $\ua$ in $|L|$,  and $V^{d,\ua}(S,L)$ to be the closure of $V_0^{d,\ua}(S,L)$. Then $V_0^{d,\ua}(S,L)$ is smooth, $V^{d,\ua}(S,L)$ is irreducible, and their codimensions in $|L|$ are both $\codim(\ua)=\sum_{i=1}^{l(\ua)} \tau(\alpha_i)$.  \end{theo}

\begin{proof}
Let $\cC$  be the universal family of curves in $|L|$ with projection $\cC \to |L|$,  and let $\cC^{[n]}$ be the relative Hilbert scheme of $n$ points of the family. 
There are  natural projections from $\cC^{[n]}$ to $\Hn{S}$ and from $\cC^{[n]}$ to $|L|$. 

Consider the following commutative diagram: 

\begin{center}        
$$\xymatrix{
|L| &  \cC^{[N(\ua)]} \ar[l] \ar[r]& \Hilb{S}{N(\ua)}\\
V^{d,\ua}(S,L)\ar[u]& \Sigma_0 \ar[lu] \ar[r] \ar[u]&   S^0(\ua) \ar[u]&
}$$
\end{center}

The right part of the diagram is Cartesian; i.e. $\Sigma_0$ is defined to be $\cC^{[N(\ua)]}\underset{\Hilb{S}{N(\ua)}}{\times} S^0(\ua)$ and  therefore is the union of $\{\xi \times C\,|\,  \xi \in S^0(\ua), \, C \text{ is a curve in } |L|,\,\xi\subset C\}$. 

If $L$ is $(N(\ua)-1)$-very ample, fibers of $\cC^{[N(\ua)]} \to \Hilb{S}{N(\ua)}$ are all projective spaces of  constant dimension, so $\Sigma_0 \to S^0(\ua)$ is a projective bundle. 
Because  $S^0(\ua)$ is smooth and connected \cite{Gr},  $\Sigma_0$ is also smooth connected and thus irreducible.

The image Im($\Sigma_0$)$\subset |L|$ is the collection of curves which contain closed subschemes in $S^0(\ua)$. 
If a curve has singularity type $\ua$, it belongs to  Im($\Sigma_0$), which implies $V_0^{d,\ua}(S,L)\subset \text{Im}(\Sigma_0) $ and  $V^{d,\ua}(S,L)\subset \overline{\text{Im}(\Sigma_0)} $.
The complement of  $V_0^{d,\ua}(S,L)$ in  $\text{Im}(\Sigma_0)$ are curves that contain closed subschemes in $S^0(\ua)$ but do not have singularity type $\ua$. 
In the proof of Proposition \ref{prop:daL}, we showed that these curve must contain closed subschemes in 
\begin{align*}
\tilde{S}^0(\ua)&= \left\{\displaystyle \Cxy/\langle g_1\fm, \fm^{k(\alpha_1)+1}\rangle
\cup  \left(\coprod_{i=2}^{l(\ua)}\eta_i \right)  \right\}\subset \Hilb{S}{N(\ua)+1}, \, \\
S'(\ua)&=\left\{\displaystyle \coprod_{i=1}^{l(\ua)+1} \eta_i  \,\left|\,\,
\coprod_{i=1}^{l(\ua)} \eta_i \in S^0(\ua) \text{ and } \eta_{l(\ua)+1}\cong
\Cxy/\fm^2 \right. \right\}\subset \Hilb{S}{N(\ua)+3},  
\end{align*}
or finite subsets of $\Hilb{S}{N(\ua)+1}$ similar to $\tilde{S}^0(\ua)$ if the singularity at $x_i$ is not $\alpha_i$. 
Simple calculation shows that the complement of $V_0^{d,\ua}(S,L)$ in  $\text{Im}(\Sigma_0)$ has dimension strictly less than $ \dim V_0^{d,\ua}(S,L)$. It follows that  $V_0^{d,\ua}(S,L)$ is dense in $\text{Im}(\Sigma_0)$ and its closure $V^{d,\ua}(S,L)$ is dense in $\overline{\text{Im}(\Sigma_0)}$. 
Since Im($\Sigma_0$) is irreducible, so is $\overline{\text{Im}(\Sigma_0)}$, therefore $V^{d,\ua}(S,L)$ must be the only irreducible component of $\overline{\text{Im}(\Sigma_0)}$. 
We conclude that $V^{d,\ua}(S,L)=\overline{\text{Im}(\Sigma_0)}$ is irreducible.

Next, we prove $V_0^{d,\ua}(S,L)$ is smooth and of codimension $\codim(\ua)$. 
Let $s\in |L|$ define a curve $C$ with singularity type $\ua$.  The germ of $|L|$ at $s$ maps to the miniversal deformation space $\mathsf{Def}$ of the singularities of $C$. The corresponding map of tangent spaces  $T_s|L|=H^0(L)/\langle s\rangle \to H^0(L\otimes \cO_C/J)$ is onto if $h^0(L\otimes \cO_C/J)=\sum_{i=1}^{l(\ua)} \tau(\alpha_i)\leq N(\ua)+3$, because $L$ is $(N(\ua+2)$-very ample. 
For every $\alpha \in \ua$, let $f_{\alpha}$ be a germ that defines $\alpha$ at the origin.  Since $f_{\alpha}$ is (analytically) $k(\alpha)$-determined, $ \fm^{k(\alpha)+1} \subseteq \fm J(f_{\alpha}) + \langle f_{\alpha} \rangle$ (Theorem \ref{thm:GLS}).  Therefore,

\begin{align*}\tau(\alpha)=&\dim_{\CC} \Cxy/\langle f_{\alpha},  J(f_{\alpha}) \rangle \leq 
\dim_{\CC} \Cxy/\langle f_{\alpha},  \fm J(f_{\alpha}) \rangle \\ \leq &
\dim_{\CC} \Cxy/\langle f_{\alpha}, \fm^{k(\alpha)+1} \rangle= N(\ua)<N(\ua)+3. \end{align*}

Therefore the map to miniversal deformation space is a smooth map and thus $V_0^{d,\ua}(S,L)$  is smooth and of codimension $\codim(\ua)=\sum_{i=1}^{l(\ua)} \tau(\alpha_i)$ in $|L|$.
\end{proof}

\bibliography{../../../mybib}
\bibliographystyle{amsplain}

\end{document}